\documentclass[11pt]{amsart}
\usepackage{amscd,amsmath,latexsym,amsthm,amsfonts,amssymb,graphicx,color,geometry,hyperref,bm}

\usepackage[utf8]{inputenc}
\usepackage[T1]{fontenc}

\usepackage[norefs,nocites]{refcheck}
\hypersetup{
colorlinks=true,
linkcolor=blue,
citecolor=cyan
}

\allowdisplaybreaks

\newtheorem{theorem}{Theorem}[section]
\newtheorem{proposition}[theorem]{Proposition}
\newtheorem{corollary}[theorem]{Corollary}

\newtheorem{remark}[theorem]{Remark}

\newtheorem{definition}[theorem]{Definition}
\numberwithin{equation}{section}

\newcommand{\dlim}{\displaystyle\lim}

\begin{document}

\title{Closed Graph Theorem for Linear Random Operators}


\begin{abstract}
In this work, which was inspired by the article \cite{TGF} by M. V. Velasco and A. R. Villena, we obtain a characterization for probably continuous operators and show that the probability of a linear random operator being continuous coincides with the probability of having a closed graph. Furthermore, we laid the foundations to start a new branch of mathematics, namely, Random Analysis.  
\end{abstract}

\maketitle

\begin{center}
	\textbf{Author: Kleber Soares Câmara}
\end{center}

\tableofcontents

\author{C\^amara., K. S.}\newline\indent
\address{Departamento de Matem\'{a}tica,\newline\indent
	Universidade Federal Rural do Semi-árido,\newline\indent
	59.625-900, Mossoró -- RN, Brazil.}\newline\indent
\email{kleber.soares@ufersa.edu.br}

\section{Introdução}

The study of random variables that take values in Banach spaces is very important in the theory of random equations \cite{Bharucha, Skorohold}, where stochastic continuity is an essential property. Stochastically continuous random linear operators behave like continuous in some sense that can be quantified with some precision. We will see that the "stochastic size" of the separation subspace of a linear random operator designates its continuity.


\section{Random variables}

Let's look at some new definitions and others that can be found in \cite{TGF, Bharucha}. 

Classical probability theory introduces the concept of random variables as being measurable functions that take on values in the set $\mathbb R$ of the real numbers. Generally speaking, random variables that take on values in a Banach space are measurable functions; 
Throughout this work, we will consider a fixed probability space $(\Omega, \mathcal{A}, \mathbb{P}),$ where $\Omega$ is an arbitrary set, $\mathcal{A}$ a $\sigma$--algebra of subsets of $\Omega$ and $\mathbb{P}$ a measure of probability. 

Let $(Y, \mathfrak B)$ be a measurable space where $Y$ is a Banach space and $\mathfrak B$ is the $\sigma$--algebra of all Borel subsets of $Y.$

\begin{definition}
	A map $x:\Omega\to Y$ is said to be a random variable with values in $Y$ if the inverse image of every Borel set $B,$ under the map $x,$ belongs to $\mathcal{A},$ i.e. if $x^{-1}(B)\in\mathcal A$ for all $B\in\mathfrak B.$ 
\end{definition}

In other words, a random variable valued in $Y$ is a measurable Borel function valued in the Banach space $Y$; When $Y=\mathbb{R},$ the above definition matches the definition of ordinary random variable. 

%
%
%
%
%
A randomization $\mathcal{R}_{\Omega}(Y)$ ($\mathcal{R}(Y)$ if there is no confusion) is the linear space of all random variables $Y- $valued over $\left(\Omega, \mathcal{A}, \mathbb{P}\right)$ (where matching elements are almost certainly equivalent). As in \cite{Velasco}, we consider that all randomization is equipped with the topology of convergence in probability, which is metrizable (see \cite[Chapter 2]{Ledoux}). 

\section{Characterizing the probable continuity of linear random operators}

Given $y\in \mathcal{R}_{\Omega}(Y)$ and given $\tau\in\mathbb{R},$ we use the notation 
$$\left[y\ge\tau\right]=\{\omega\in\Omega; y(\omega)\ge\tau\},$$
as well as
$$\left[\|y\|\ge\tau\right]=\{\omega\in\Omega; \|y(\omega)\|\ge\tau\}.$$
Analogous notations are used when replacing the aforementioned $\ge$ with $\le, =, >$ or $<.$ Also, given $A, B\in\mathcal{A},$ we use the notation $$A,B=A\cap B.$$

\begin{definition}
	A sequence $(y_n)$ in $\mathcal{R}(Y)$ is said to converge in probability to an element $y$ in $\mathcal{R}(Y)$ when
	$$\lim_{n\to \infty}\mathbb{P}\left[\|y_n-y\|\ge\tau\right]=0,\quad\forall \tau>0$$
	or, more specifically, when for all $\tau>0$ and $ 0<\varepsilon \le 1,$ there is $N\in\mathbb{N}$ such that
	\[
	\mathbb{P}\left[\|y_n-y\|\ge\tau\right]<\varepsilon
	\]
	whenever it is $n>N.$
\end{definition}

It is important to note that, due to the equality 
$$\mathbb{P}\left[\|y_n-y\|<\tau\right]+\mathbb{P}\left[\|y_n-y\|\ge\tau\right]=1,$$
it follows that $\lim_{n\to \infty}\mathbb{P}\left[\|y_n-y\|\ge\tau\right]=0,\quad\forall \tau>0,$ equivalent to 
$$\lim_{n\to \infty}\mathbb{P}\left[\|y_n-y\|<\tau\right]=1,\quad\forall \tau>0;$$
In other words, for all $\tau>0$ and $ 0<\varepsilon \le 1,$ there is $N\in\mathbb{N}$ such that
\[
\mathbb{P}\left[\|y_n-y\|<\tau\right]\ge\varepsilon
\]
for all $n>N.$
\begin{definition}
	An operator $T:X\to \mathcal{R}(Y),$ called the random operator of $X$ on $Y,$ is said to be linear if, for all constant $\alpha, \beta,$ 
	$$\mathbb{P}\left[T(\alpha x+\beta y)=\alpha T(x)+\beta T(y)\right]=1, \quad\forall x,y\in X.$$
\end{definition}
In particular, note that $\mathbb{P}\left[T(0)=0\right]=1,$ since $\mathbb{P}\left[T(0 x+0 y)=0 T(x)+0 T(y)\right]=1$ for all $x,y\in X.$ Also note that the inclusion $$\left[T(\alpha x+\beta y)=\alpha T(x)+\beta T(y)\right]\subset \left[\|T(\alpha x+\beta y)\|=\|\alpha T(x)+\beta T(y)\|\right]$$ will give us $\mathbb{P}\left[\|T(\alpha x+\beta y)\|=\|\alpha T(x)+\beta T(y)\|\right]=1$ always that the random operator $T$ is linear. Finally, it's important to remember that given $A, B\in \mathcal{A}$ with, say, $\mathbb{P}(B)=1$ and using the notation $A,B=A\cap B,$ we have
$$\mathbb{P}\left(A\right)=\mathbb{P}\left(A,B\right),$$ 
since
$1=\mathbb{P}\left(B\right)=\mathbb{P}\left(A\cup B\right)=\mathbb{P}\left(A\right)+\mathbb{P}\left(B\right)-\mathbb{P}\left(A, B\right)=\mathbb{P}\left(A\right)+1-\mathbb{P}\left(A, B\right).$

We adopt the definition of stochastically continuous operator given in \cite{BHARUCHA-REID}, where the 
spaces $X$ and $Y$ as well as giving some more pertinent definitions concerning probable continuity, whose generality allows us to obtain stochastic continuity as one of its particular cases.

\begin{definition}\label{defII}
	Let $T:X\to \mathcal{R}_{\Omega}(Y)$ be a linear random operator. If there is 
	\item [1)] for all $\tau>0$ and $0<\varepsilon< \alpha,$ there is $\delta>0$ such that $\mathbb{P}\left[\|T(x)-T(x_0)\|<\tau\right]>\varepsilon$ whenever $x\in X$ and $\|x-x_0\|<\delta.$ In this case, we say that $T$ is probably continuous on $x_0$ with probability $\ge\alpha$ and we will denote this by
	$$\lim_{x\to x_0}\mathbb{P}\left[\|T(x)-T(x_0)\|<\tau\right]\ge\alpha.$$
	$T$ will be said to be probably continuous (with probability $\ge\alpha$) if it is said to be continuous at all $x_0\in X;$
	\item [2)] for every $0<\varepsilon<\alpha,$ there is $M_\varepsilon>0$ such that $\mathbb{P}\left[\|T(x)\|\le M_\varepsilon\right]>\varepsilon,$ for all $x\in B_{X}=\{x\in X; \|x\|\le 1\},$ we will say that T is probably bounded (with probability $\ge\alpha$); 
	\item  [3)] for every $0<\varepsilon<\alpha,$ there is $M_\varepsilon>0$ such that
	$$\mathbb{P}\left[\|T(x)-T(y)\|\le M_\varepsilon\|x-y\|\right]>\varepsilon, \quad\forall x,y\in X,$$
	we will say that T is probably lipschitzian (with probability $\ge\alpha$); 
	\item  [4)] given $\tau>0$ and $0<\varepsilon<\alpha,$ there is $\delta>0$ such that
	$$\mathbb{P}\left[\|T(x)-T(y)\|< \tau\right]>\varepsilon, \text{ whenever } x,y\in X \text{ and } \|x-y\|<\delta,$$
	we will say that $T$ is uniformly continuous probably (with probability $\ge\alpha$).
\end{definition}

\begin{definition}
	We will say that the operator $T$ is stochastically continuous (bounded, lipschitzian or uniformly continuous) if $\alpha = 1$ satisfies item $(1)$ of Definition \ref{defII} (respectively, satisfies items $(2), (3)$ or $(4)$ of the same definition). 
\end{definition}

In item $1$ of the definition above, if $0<\alpha<1$ is the largest number for which the limit $lim_{x\to x_0}\mathbb{P}\left[\|T(x)-T(x_0)\|<\tau\right]\ge\alpha,$ i.e. if $\alpha$ is such that for every $\alpha<\gamma\le1,$ there are $\tau > 0$ and $\alpha\le\varepsilon<\gamma$ such that for every $\delta>0$ there is $x\in X$ satisfying
$$\mathbb{P}\left[\|T(x)-T(x_0)\|<\tau\right]\le\varepsilon$$
with $\|x-x_0\|<\delta,$ we will say that $T$ is probably continuous with probability $\alpha$ and we will write
$$\lim_{x\to x_0}\mathbb{P}\left[\|T(x)-T(x_0)\|<\tau\right]=\alpha.$$
Thus, $T$ will be stochastically continuous if and only if,
$$\lim_{x\to x_0}\mathbb{P}\left[\|T(x)-T(x_0)\|<\tau\right]=1$$
for all $\tau>0,$ that is, if and only if $T$ is probably continuous with probability $1.$

Also note that item $(1)$ of the above definition can be reformulated due to the equality 
$$\mathbb{P}\left[\|T(x)-T(x_0)\|<\tau\right]+\mathbb{P}\left[\|T(x)-T(x_0)\|\ge\tau\right]=1,$$ for which
$$\mathbb{P}\left[\|T(x)-T(x_0)\|<\tau\right]>\varepsilon \Leftrightarrow \mathbb{P}\left[\|T(x)-T(x_0)\|\ge\tau\right]<1-\varepsilon$$
and consequently 
\begin{eqnarray}\label{eqv}
\dlim_{x\to x_0}\mathbb{P}\left[\|T(x)-T(x_0)\|<\tau\right]\ge\alpha \Leftrightarrow \dlim_{x\to x_0}\mathbb{P}\left[\|T(x)-T(x_0)\|\ge\tau\right]\le 1-\alpha;
\end{eqnarray}
thus, $T$ is probably continuous in $x_0\in X$ with probability $\ge\alpha$ if and only if $$\dlim_{x\to x_0}\mathbb{P}\left[\|T(x)-T(x_0)\|\ge\tau\right]\le 1-\alpha.$$ It is important to note that the equivalence \eqref{eqv} holds, in particular, for the sequential limit, that is,
$$\dlim_{n\to \infty}\mathbb{P}\left[\|T(x_n)-T(x_0)\|<\tau\right]\ge\alpha$$
equals
$$\dlim_{n\to \infty}\mathbb{P}\left[\|T(x_n)-T(x_0)\|\ge\tau\right]\le 1-\alpha$$
for every sequence $\left(x_n\right)_{n=1}^\infty$ in $X$ converging to $x_0.$ This last limit, in turn, means that for all $\tau>0$ and $0<\varepsilon< \alpha,$ $\mathbb{P}\left[\|T(x_n)-T(x_0)\|\ge\tau\right]<1-\varepsilon$ whenever $x_n\in X$ and $\|x_n-x_0\|<\delta.$

\begin{remark}[Duality]
	Comments similar to those made for item $(1)$ of Definition \ref{defII}, related to the exchange of $\mathbb{P}\left[\cdot\right]\ge\varepsilon$ for $\mathbb{P}\left[\cdot^c\right]<1-\varepsilon,$ can also be done for items $(2),(3)$ and $(4)$ of the same Definition.
\end{remark}

\begin{remark}\label{propC.O.P.}
	Let $T:X\to \mathcal{R}_{\Omega}(Y)$ be a linear random operator. Then $T$ is probably continuous at the origin with probability $\ge\alpha$, that is,
	$$\dlim_{x\to 0}\mathbb{P}\left[\|T(x)-T(0)\|<\tau\right]\ge\alpha, \textrm{ for all } \tau>0,$$
	if, and only if,
	$$\dlim_{x\to 0}\mathbb{P}\left[\|T(x)\|<\tau\right]\ge\alpha, \textrm{ for all } \tau>0.$$
	In fact, it suffices to show that 
	$$\mathbb{P}\left[\|T(x)-T(0)\|<\tau\right]=\mathbb{P}\left[\|T(x)\|<\tau\right].$$
	Since $\|T(x)\|-\|T(0)\|\le\|T(x)-T(0)\|,$ for all $x\in X,$ and $\mathbb{P}\left[\|T(0)\|=0\right]=1,$ we have
	\begin{eqnarray*}
		\mathbb{P}\left[\|T(x)-T(0)\|<\tau\right] &\le& \mathbb{P}\left[\|T(x)\|-\|T(0)\|<\tau\right]\\ &=& \mathbb{P}\left[\|T(x)\|-\|T(0)\|<\tau, \|T(0)\|=0\right]\\ &=& 
		\mathbb{P}\left[\|T(x)\|-0<\tau, \|T(0)\|=0\right]\\ &=&
		\mathbb{P}\left[\|T(x)\|<\tau\right],
	\end{eqnarray*}
	that is,  $\mathbb{P}\left[\|T(x)-T(0)\|<\tau\right]\le\mathbb{P}\left[\|T(x)\|<\tau\right].$ 
	
	Conversely, in view of the inequality $\|T(x)-T(0)\|-\|T(0)\|\le \|T(x)\|,$ we have 
	\begin{eqnarray*}
		\mathbb{P}\left[\|T(x)\|<\tau\right] &\le& \mathbb{P}\left[\|T(x)-T(0)\|+\|T(0)\|<\tau\right]\\ &\le& \mathbb{P}\left[\|T(x)-T(0)\|<\tau\right]
	\end{eqnarray*}
	i.e. $\mathbb{P}\left[\|T(x)\|<\tau\right]\le\mathbb{P}\left[\|T(x)-T(0)\|<\tau\right]$ and consequently $\mathbb{P}\left[\|T(x)-T(0)\|<\tau\right]=\mathbb{P}\left[\|T(x)\|<\tau\right].$
\end{remark}


It is important to note that when there is $\alpha=1$ satisfying the theorem above, we conclude that $T$ is stochastically continuous at the origin, that is,
$$\lim_{x\to 0}\mathbb{P}\left[\|T(x)-T(0)\|<\tau\right]=1, \textrm{ for all } \tau>0,$$
if, and only if,
$$\lim_{x\to 0}\mathbb{P}\left[\|T(x)\|<\tau\right]=1, \textrm{ for all } \tau>0.$$

Note that, taking into account the equivalences of type (\ref{eqv}), we have, for all $\tau>0,$
$$ \dlim_{x\to 0}\mathbb{P}\left[\|T(x)-T(0)\|\ge\tau\right]\le 1-\alpha$$
if, and only if, 
$$\dlim_{x\to 0}\mathbb{P}\left[\|T(x)\|\ge\tau\right]\le 1-\alpha,$$
and when $\alpha=1,$ remains the conclusion that $T$ is stochastically continuous.

The following result provides us with a characterization for probably continuous random operators, as well as, in particular, when $\alpha=1$, characterizes stochastically continuous operators.

\begin{theorem}\label{Theorem2.P}
	Given $0<\alpha\le1$ and a linear random operator $T:X\to \mathcal{R}_\Omega(Y)$, the following statements are equivalent.
	\begin{itemize}
		\item[i)] $T$ is probably Lipschitzian with probability $\ge\alpha$.
		\item[ii)] $T$ is probably uniformly continuous with probability $\ge\alpha$.
		\item[iii)] $T$ is probably continuous with probability $\ge\alpha$.
		\item[iv)] $T$is probably continuous at some point in $X$ with probability $\ge\alpha$.
		\item[v)] $\dlim_{x\to 0}\mathbb{P}\left[\|T(x)\|<\tau\right]\ge\alpha, \text{ for all } \tau>0,$ i.e. $T$ isl probably continuous at the origin with probability $\ge\alpha$.
		\item[vi)] $T$ is probably bound with probability $\ge\alpha$.
		\item[vii)] For every $0<\varepsilon <\alpha$ there is a constant $M_\varepsilon\ge 0$ such that 
		$$\mathbb{P}\left[\|T(x)\|\le  M_\varepsilon \|x\| \right] >\varepsilon \text{ for all }x \text{ in } X.$$
	\end{itemize}
\end{theorem}

\begin{proof}
	
	$i) \Rightarrow ii) $ In fact, given $\tau>0$ and $0<\varepsilon<\alpha,$ and assuming that $T$ is probably Lipschitzian with probability $\ge\alpha$, that is, that for all $0<\varepsilon<\alpha$ exists $M_\varepsilon>0$ such that $\mathbb{P}\left[\|T(x)-T(y)\|\le M_\varepsilon\|x-y\|\right]>\varepsilon$ for all $x,y\in X,$ just take $\delta=\tau/M_\varepsilon.$ Indeed, if $x,y\in X$ and $\|x-y\|<\delta=\tau/M_\varepsilon,$ then
	\begin{eqnarray*}
		\mathbb{P}\left[\|T(x)-T(y)\|<\tau\right]&=&\mathbb{P}\left[\|T(x)-T(y)\|< M_\varepsilon\delta\right]\\ &\ge& \mathbb{P}\left[\|T(x)-T(y)\|\le M_\varepsilon\|x-y\|\right]>\varepsilon,
	\end{eqnarray*}
	that is, $T$ is uniformly continuous probably with probability $\ge\alpha$.
	
	$ii) \Rightarrow iii)$ Just set $y=x_0$ Just set $4$ of Definition \ref{defII}.
	
	$iii) \Rightarrow iv)$ Immediate.
	
	$iv) \Rightarrow v)$ 
	Assuming $T$ probably continuous in $x_0\in X$ with probability $\ge\alpha,$ given $\tau>0$ and $0<\varepsilon<\alpha,$ we get $\mathbb{P}\left[\|T(x)-T(x_0)\|<\tau\right]>\varepsilon$ whenever $x\in X$ and $\|x-x_0\|<\delta.$ Taking $x\in X$ such that $\|x-0\|=\|x\|<\delta,$ we have $\|(x+x_0)-x_0\|<\delta,$ hence $\mathbb{P}\left[\|T(x+x_0)-T(x_0)\|<\tau\right]>\varepsilon.$ So, since
	$$\mathbb{P}\left[\|T(x+x_0)-T(x_0)\|=\|T(x)+T(x_0)-T(x_0)\|\right]=1,$$
	we have
	\begin{eqnarray*}
		&\mathbb{P}\left[\|T(x)+T(x_0)-T(x_0)\|<\tau\right]\\
		=&\mathbb{P}\left[\|T(x)+T(x_0)-T(x_0)\|<\tau, \|T(x+x_0)-T(x_0)\|=\|T(x)+T(x_0)-T(x_0)\|\right]\\
		&=\mathbb{P}\left[\|T(x+x_0)-T(x_0)\|<\tau\right]
	\end{eqnarray*}
	by what we can write
	\begin{eqnarray*}
		\mathbb{P}\left[\|T(x)\|<\tau\right]&=&\mathbb{P}\left[\|T(x)+T(x_0)-T(x_0)\|<\tau\right]\\ &=& \mathbb{P}\left[\|T(x+x_0)-T(x_0)\|<\tau\right]>\varepsilon,
	\end{eqnarray*}
	i.e. $\dlim_{x\to 0}\mathbb{P}\left[\|T(x)\|<\delta\right]\ge\alpha, \text{ for all } \tau>0.$
	
	$v) \Rightarrow vi)$ If $T$ is probably continuous at the origin with probability $\ge\alpha$, that is, if $\dlim_{x\to 0}\mathbb{P}\left[\|T(x)\|<\tau\right]\ge\alpha, \text{ for all } \tau>0,$ so given $\tau>$ and $0<\varepsilon<\alpha,$ exists
	$\delta>0$ such that $\mathbb{P}\left[\|T(x)\|<\tau\right]>\varepsilon$ whenever $\|x\|\le\delta.$ Fixed $\tau=C>0, \delta$ will only be a function of $\varepsilon.$ As
	$$\mathbb{P}\left[\|T(\delta x/\delta)\|=\|\delta T(x/\delta)\|\right]=1,$$
	we have
	\begin{eqnarray*}
		\varepsilon<\mathbb{P}\left[\|T(x)\|<C\right]
		&=&\mathbb{P}\left[\|T(\delta x/\delta)\|<C, \|T(\delta x/\delta)\|=\|\delta T(x/\delta)\|\right]\\ 
		&=&\mathbb{P}\left[\|\delta T(x/\delta)\|<C, \|T(\delta x/\delta)\|=\|\delta T(x/\delta)\|\right]\\
		&=& \mathbb{P}\left[\|T(x/\delta)\|<C/\delta\right],
	\end{eqnarray*}
	whenever $\|x\|\le\delta$ and, in this case,  $\|x/\delta\|\le 1,$ so, writing $C/\delta=M_\varepsilon,$ we conclude that $\mathbb{P}\left[\|T(x)\|\le M_\varepsilon\right]\ge\mathbb{P}\left[\|T(x)\|<M_\varepsilon\right]>\varepsilon$ for all $x\in B_{X}.$
	
	$vi) \Rightarrow vii)$ Suppose now that for every $0<\varepsilon<\alpha,$ there exists $M_\varepsilon$ such that $\mathbb{P}\left[\|T(x)\|\le M_\varepsilon\right]>\varepsilon$ for all $x\in B_{X}.$ So for all  $x\in X\setminus\{0\},$ we have $x/\|x\|\in B_{X},$ thus, with a reasoning similar to that used in the previous item, we concluded that
	\begin{eqnarray*}
		\varepsilon<\mathbb{P}\left[\|T(x/\|x\|)\|\le M_\varepsilon\right]&=&\mathbb{P}\left[\|T(x)\|\le M_\varepsilon\|x\|\right]\text{ for all }x\in X.
	\end{eqnarray*}
	The case $x=0$ is immediate.
	
	$vii) \Rightarrow i) $ As $\mathbb{P}\left[\|T(x)-T(y)\|=\|T(x-y)\|\right]=1,$ and assuming $(vii)$, we can write
	\begin{eqnarray*}
		\mathbb{P}\left[\|T(x)-T(y)\|\le M_\varepsilon\|x-y\|\right]:=\mathbb{P}\left(\{\omega\in\Omega; \|T(x)(\omega)-T(y)(\omega)\|\le M_\varepsilon\|x-y\|\| \}\right)=\\ \mathbb{P}\left(\{\omega\in\Omega; \|T(x)(\omega)-T(y)(\omega)\|\le M_\varepsilon\|x-y\|\}, \{\omega\in\Omega; \|T(x)(\omega)-T(y)(\omega)\|=\|T(x-y)(\omega)\| \}\right)=\\ \mathbb{P}\left(\{\omega\in\Omega; \|T(x-y)\|\le M_\varepsilon\|x-y\|\}, \{\omega\in\Omega; \|T(x)(\omega)-T(y)(\omega)\|=\|T(x-y)(\omega)\| \}\right)=\\ \mathbb{P}\left[\|T(x-y)\|\le M_\varepsilon\|x-y\|\right]>\varepsilon,
	\end{eqnarray*}
	that is, for every $0<\varepsilon<\alpha,$ there is $M_\varepsilon>0$ such that
	$$\mathbb{P}\left[\|T(x)-T(y)\|\le M_\varepsilon\|x-y\|\right]>\varepsilon, \quad\forall x,y\in X.$$
	
\end{proof}

In particular, when $\alpha=1,$ the Theorem \ref{Theorem2.P} provides a characterization of stochastically continuous random operators.

\section{Randomized version of the Closed Graph Theorem}

\subsection{Conditionally continuous linear random operators}

Given a linear random operator $T:X\to \mathcal{R}_\Omega(Y)$ and a measurable subset $\Omega'$ with $\mathbb{P}\left[\Omega'\right]>0$, we can consider $\Omega'$ as a new probability space with the structure inherited from $\Omega$.  In this case, the probability induced on $\Omega'$ is the conditional probability $\mathbb P'$ given by 
$$\mathbb P'(\hat{\Omega})=\frac{\mathbb P(\hat{\Omega})}{\mathbb P(\Omega')},\quad\forall\hat{\Omega}\in\mathcal{A'}, $$
where $\mathcal{A'}$ is the $\sigma$--algebra induced in $\mathcal{A}$ by $\Omega'.$ Thus, for each random variable $y$ in $\mathcal{R}_{\Omega}(Y)$, we will consider its restriction $y/\Omega'$ and, in these terms, we present the following definition.

\begin{definition}
	Let $T(x)/\Omega'$ be the restriction of the random variable $T(x)$ to the subspace $\Omega'$, with $\mathbb P (\Omega')> 0$. The operator $T/\Omega': X\to \mathcal{R}_{\Omega'}(Y)$, given by
	$$T/\Omega'(x)=T(x)/\Omega',$$
	is named \textit{conditional operator} of $T.$  	
\end{definition}

\begin{proposition}\label{propP}
	Let $0<\alpha\le1.$ A linear random operator $T$ is \textit{probably continuous} with probability $\ge\alpha$ if and only if there is a measurable subset $\Omega'\subset\Omega$ with $\mathbb P\left(\Omega'\right)\ge\alpha$ such that the conditional operator $T/\Omega'$ is stochastically continuous (on probability $\mathbb P'$). 
\end{proposition}

\begin{proof}
		
	Assuming that $T$ is probably continuous with probability $\ge\alpha,$ that is, that 
	$$\mathbb{P}\left[\|T(x)\|\le  M_\varepsilon \|x\| \right] >\varepsilon,$$ 
	for all $0<\varepsilon<\alpha$ and $x\in X,$ let us consider a sequence $(\varepsilon_n),~0<\varepsilon_n<\alpha$ such that $\varepsilon_n\to \alpha$ and, for each $n\in\mathbb N$ and $x\in X,$ let's define 
	$$\Omega_x^{\varepsilon_n}=\left[\|T(x)\|\le  M_{\varepsilon_n} \|x\| \right],$$
	that is, $\Omega_x^{\varepsilon_n}=\{\omega\in\Omega;\|T(x)(\omega)\|\le  M_{\varepsilon_n} \|x\| \}.$ So, by hypothesis, we have $\mathbb P \left(\Omega_x^{\varepsilon_n}\right)>\varepsilon_n$ for all $x\in X.$ Also writing
	$$\Omega_x^{n}=\bigcup_{k=1}^n\Omega_x^{\varepsilon_k}$$
	we get, for all $x\in X$ and $n\in\mathbb N,$
	$$\mathbb P\left(\Omega_x^{n}\right)>\varepsilon_n \text{ and } \Omega_x^{n}\subset\Omega_x^{n+1}$$ 
	so, for each $x\in X,$ a new event is defined, denoted by $\dlim_{n\to \infty}\Omega_x^{n}$, to know
	$$\dlim_{n\to \infty}\Omega_x^{n}=\bigcup_{k=1}^{\infty}\Omega_x^{k}$$
	who has the property
	$$\mathbb P\left(\bigcup_{k=1}^{\infty}\Omega_x^{k}\right):=\mathbb P \left(\dlim_{n\to \infty}\Omega_x^{n}\right)=\dlim_{n\to \infty}\mathbb P\left(\Omega_x^{n}\right)\ge\dlim_{n\to \infty}\varepsilon_n=\alpha.$$
	So, fixed $x\in X,$ just take $\Omega'=\bigcup_{k=1}^{\infty}\Omega_x^{k}$ and we will have 
	
	$$\mathbb{P}\left[\|T/\Omega'(x)\|\le  M_\varepsilon \|x\| \right] =\mathbb P \left[\Omega'\right]>\mathbb P \left[\Omega'\right]\varepsilon,\quad\forall~ 0<\varepsilon<1, $$
	or yet,
	$$\mathbb{P}'\left[T/\Omega'\right]:=\frac{\mathbb{P}\left[\|T(x)\|\le  M_\varepsilon \|x\| \right]}{\mathbb P \left[\Omega'\right]} >\varepsilon,\quad\forall~ 0<\varepsilon<1,$$
	i.e. there exists a measurable subset $\Omega'\subset\Omega$ with $\mathbb P\left(\Omega'\right)\ge\alpha$ such that the conditional operator $T/\Omega'$ is stochastically continuous in the conditional probability $\mathbb P'$. 
	
	Conversely, in view of item $(vii)$ of Theorem \ref{Theorem2.P},  to say that $T/\Omega'$ is stochastically continuous in the conditional probability $\mathbb P',$ means to say that for every $0<\varepsilon <1$ there is a constant $M_\varepsilon\ge 0$ such that, $ \text{for all }x \text{ in } X,$
	
	\begin{eqnarray*}
		\mathbb{P}'\left[\|T/\Omega'(x)\|\le  M_\varepsilon \|x\| \right] >\varepsilon 
		&\Leftrightarrow& \frac{\mathbb{P}\left[\|T(x)/\Omega'\|\le  M_\varepsilon \|x\| \right]}{\mathbb{P}(\Omega')} >\varepsilon\\	
		&\Leftrightarrow& \mathbb{P}\left[\|T(x)/\Omega'\|\le  M_\varepsilon \|x\| \right] >\mathbb P (\Omega') \varepsilon.
	\end{eqnarray*}
	Since
	$$\left[\|T(x)/\Omega'\|\le M_\varepsilon\|x\|\right]\subset\left[\|T(x)\|\le M_\varepsilon\|x\|\right],$$
	we have
	$$\mathbb{P}\left[\|T(x)/\Omega'\|\le  M_\varepsilon \|x\| \right] >\mathbb P (\Omega') \varepsilon 
	\Rightarrow \mathbb{P}\left[\|T(x)\|\le  M_\varepsilon \|x\| \right] >\mathbb{P}(\Omega')\varepsilon,$$
	hence
	$$\mathbb{P}'\left[\|T/\Omega'(x)\|\le  M_\varepsilon \|x\| \right] >\varepsilon 
	\Rightarrow \mathbb{P}\left[\|T(x)\|\le  M_\varepsilon \|x\| \right] >\mathbb{P}(\Omega')\varepsilon\ge\alpha\varepsilon.$$
	As $0<\alpha\varepsilon<\alpha$ for all $0<\varepsilon<1,$ it follows that $T$ is probably continuous with probability $\ge\alpha.$
\end{proof}

\begin{definition}
	For a probably continuous linear random operator $T:X\to\mathcal{R}_\Omega(Y),$ we define the real number
	$$\alpha_T=\sup\{\alpha;\text{ if } 0<\varepsilon<\alpha, \exists ~ M_{\varepsilon}\ge 0 \text{ such that } \mathbb{P}\left[\|T(x)\|\le  M_{\varepsilon} \|x\| \right] >\varepsilon~ \forall ~x\in X\}.$$ 
	When $T$ is probably not continuous, we define $\alpha_T=0.$ We call $\alpha_T$ \textit{probability of $T$ being continuous}.
\end{definition} 

Note that when T is stochastically continuous, we have $\alpha_T=1.$ On the other hand, Corollary \ref{alpha} below shows us, among other things, that if $\alpha_T=1,$ $T$ is stochastically continuous. This same result also justifies the fact that we call $\alpha_T$ \textit{probability of $T$ being continuous}.

Next, the sets $\Omega'$ and $\Omega'_0$ represent certain measurable subsets of $\Omega,$ as explained in the result below:

\begin{corollary}\label{alpha}
	Let $T:X\to\mathcal{R}_\Omega(Y)$ be a linear probably continuous random operator. Then
	\begin{eqnarray*}
		\alpha_T&=&\max\{\alpha;\text{ if } 0<\varepsilon<\alpha, \exists ~ M_{\varepsilon}\ge 0 \text{ such that } \mathbb{P}\left[\|T(x)\|\le  M_{\varepsilon} \|x\| \right] >\varepsilon~ \forall ~x\in X\}\\
		&=&\max\{ \mathbb{P}\left[\Omega'\right]: T/\Omega' \text{ is stochastically continuous in probability } \mathbb{P'} \}\\
		&:=&\mathbb{P}\left[\Omega_0'\right].
	\end{eqnarray*}
	In particular, $T$ is probably continuous with probability equal to $\alpha_T,$ $T$ has a conditional operator $T/\Omega_0'$ stochastically continuous (in conditional probability) with $\mathbb P \left[\Omega_0'\right]=\alpha_T,$ and it is not possible to find a conditional operator $T/\Omega'$ stochastically continuous with $\mathbb P \left[\Omega'\right]>\alpha_T.$
\end{corollary}

\begin{proof}
	
	To show that 
	$$\alpha_T=\max\{\alpha;\text{ if } 0<\varepsilon<\alpha, \exists ~ M_{\varepsilon}\ge 0 \text{ such that } \mathbb{P}\left[\|T(x)\|\le  M_{\varepsilon} \|x\| \right] >\varepsilon~ \forall ~x\in X\},$$
	we need to show that $\mathbb{P}\left[\|T(x)\|\le  M_{\varepsilon} \|x\| \right] >\varepsilon$ for all $0<\varepsilon<\alpha_T$ and $x\in X.$ In fact, assuming the opposite, there is $\varepsilon_0,~0<\varepsilon_0<\alpha_T$ such that, for every constant $M\ge 0,$ we can find $x=x(M)$ so that
	\begin{eqnarray}\label{ep0}
	\mathbb{P}\left[\|T(x)\|\le  M \|x\| \right] \le\varepsilon_0.
	\end{eqnarray}
	On the other hand, $\alpha_T$ being the supreme of the set 
	$$\{\alpha;\text{ if } 0<\varepsilon<\alpha, \exists ~ M_{\varepsilon}\ge 0 \text{ such that } \mathbb{P}\left[\|T(x)\|\le  M_{\varepsilon} \|x\| \right] >\varepsilon~ \forall ~x\in X\},$$
	there is $\alpha_0,~\varepsilon_0<\alpha_0<\alpha_T,$ such that, for all $x\in X$ and $0<\varepsilon<\alpha_0,$ is valid
	$$\mathbb{P}\left[\|T(x)\|\le  M_{\varepsilon} \|x\| \right] >\varepsilon,$$
	which contradicts (\ref{ep0}) when $\varepsilon=\varepsilon_0.$ It follows in particular that $T$ is probably continuous with probability $\ge\alpha_T$ and $\alpha_T$ being the largest (the supreme) of the $\alpha$ for which $\mathbb{P}\left[\|T(x)\|\le  M_{\varepsilon} \|x\| \right] >\varepsilon$ for all $x\in X$ and $0<\varepsilon<\alpha,$ it follows that T has probability equal to $\alpha_T.$
	
	Let us now prove tha 
	$$\alpha_T=\max\{ \mathbb{P}\left[\Omega'\right]: T/\Omega' \text{ is stochastically continuous in probability } \mathbb{P'} \}.$$
	Since $T$ is probably continuous with probability  $=\alpha_T,$ it follows from Proposition \ref{propP} that there exists a measurable set $\Omega'_0\subset\Omega$ such that $\mathbb P \left[\Omega'_0\right]\ge\alpha_T.$ However, if it were $\mathbb P \left[\Omega'_0\right]>\alpha_T,$ there would be $\beta$ such that $\mathbb P \left[\Omega'_0\right]\ge\beta>\alpha_T$ and, by the same proposition, we would conclude that $T$ is probably continuous with probability $\ge\beta>\alpha_T,$ contradicting the first part of this corollary.
\end{proof}

\begin{corollary}\label{CorollaryE}
	For a linear random operator $T:X\to \mathcal{R}_\Omega(Y)Y$, the following statements are equivalent.
	\begin{itemize}
		\item[i)] $T$ is probably Lipschitzian with probability $\alpha_T$.
		\item[ii)] $T$ is probably uniformly continuous with probability $\alpha_T$.
		\item[iii)] $T$ is probably continuous with probability  $\alpha_T$.
		\item[iv)] $T$ is probably continuous at some point in $X$ with probability $\alpha_T$.
		\item[v)] $\dlim_{x\to 0}\mathbb{P}\left[\|T(x)\|<\tau\right]=\alpha_T, \text{ for all } \tau>0,$ i.e. $T$ is probably continuous at the origin with probability  $\alpha_T$.
		\item[vi)] $T$ is probably bound with probability $\alpha_T$.
		\item[vii)] For every $0<\varepsilon <\alpha_T$ there is a constant $M_\varepsilon\ge 0$ such that 
		$$\mathbb{P}\left[\|T(x)\|\le  M_\varepsilon \|x\| \right] >\varepsilon \text{ for all }x \text{ in } X.$$
	\end{itemize}
\end{corollary}

With the characterization given by the Theorem \ref{Theorem2.P}, as well as the Corollary \ref{CorollaryE}, we next prove that the notion of probable continuity (in particular, stochastic continuity) of a linear random operator coincides with the notion of probable continuity (respectively, stochastic continuity) sequential.

\begin{definition}
	We will say that a linear random operator $T:X\to \mathcal{R}_{\Omega}(Y)$ is probably sequentially continuous, with probability $\ge\alpha,$ when
	$$\dlim_{n\to\infty}\mathbb{P}\left[\|T(x_n)\|<\tau\right]\ge\alpha, $$
	for any $\tau>0$ and every sequence $(x_n)$ in $X$ such that $x_n\to 0.$ Precisely, this means that, given arbitrarily $\tau>0, 0<\varepsilon<\alpha$ and $(x_n)$ in $X$ such that $x_n\to 0$ when $n\to\infty,$ exists $n_0\in \mathbb{N}$ such that 
	$$\mathbb{P}\left[\|T(x_n)\|<\tau\right]>\varepsilon \text{ for all }n>n_0.$$
\end{definition}
If $\alpha$ is the largest number for which inequality in the above limit holds, that is, if given any $\gamma>\alpha$ and $n_0\in\mathbb N,$ there are $\tau>0, n>n_0$ and $\alpha\le\varepsilon<\gamma$ such that
$$\mathbb{P}\left[\|T(x_n)\|<\tau\right]\le\varepsilon,$$
we will say that $T$ is probably sequentially continuous with probability $\alpha.$ In particular, when $\alpha=1,$ we will say that $T$ is stochastically sequentially continuous, or probably sequentially continuous with probability $1.$

The following corollary shows us that a linear random operator is probably continuous, with probability $\ge\alpha$ if and only if it is probably sequentially continuous, with probability  $\ge\alpha$ and that, in addition, equality holds when $\alpha=\alpha_T$.

\begin{corollary}\label{eqseq1}
	Let $T:X\to \mathcal{R}_{\Omega}(Y)$ be a linear random operator. Then 
	$$\dlim_{x\to 0}\mathbb{P}\left[\|T(x)\|<\tau\right]\ge\alpha$$
	for all $\tau>0$
	if and only if $$\dlim_{n\to\infty}\mathbb{P}\left[\|T(x_n)\|<\tau\right]\ge\alpha$$
	for every $\tau>0$ and every sequence $(x_n)$ in $X$ such that $x_n\to 0.$ Furthermore, if $\alpha=\alpha_T,$ holds equality in both limits of the equivalence.
\end{corollary}
\begin{proof}
	
	Suppose $\dlim_{x\to 0}\mathbb{P}\left[\|T(x)\|<\tau\right]\ge\alpha$ for all $\tau>0,$ that is, that given $\tau>0$ and $0<\varepsilon<\alpha,$ exists $\delta>0$ such that $\mathbb{P}\left[\|T(x)\|<\tau\right]>\varepsilon$ whenever $\|x\|<\delta.$ Let $(x_n)$ be a sequence in $X$ such that $x_n\to 0$ when $n\to\infty.$ So there is $n_0\in\mathbb{N}$ such that $\|x_n\|<\delta$ for all $n>n_0,$ hence $\mathbb{P}\left[\|T(x_n)\|<\tau\right]>\varepsilon$ for all $n>n_0,$ i.e. $T$ is probably sequentially continuous, with probability $\ge\alpha$.
	
	To prove the converse, suppose that $T$ is probably not continuous with probability $\ge\alpha$, that is, that there are $\tau>0$ and $0<\varepsilon<\alpha$ such that, for all $\delta>0,$ exists $x\in X$ such that $\|x\|<\delta$ and $\mathbb{P}\left[\|T(x)\|<\tau\right]\le\varepsilon.$ Thus, for each $\delta=\frac{1}{n},$ we can choose  $x_n\in X$ such that $\|x_n\|<\frac{1}{n}$ and $\mathbb{P}\left[\|T(x_n)\|<\tau\right]\le\varepsilon,$ so $T$ is probably not sequentially continuous, with probability $\ge\alpha$.
	
	In particular, if $\alpha=\alpha_T,$ follows from item $(v)$ of Corollary \ref{CorollaryE} that
	$$\dlim_{x\to 0}\mathbb{P}\left[\|T(x)\|<\tau\right]=\alpha_T$$
	for all $\tau>0.$ In this case, it remains to show that $\dlim_{n\to\infty}\mathbb{P}\left[\|T(x_n)\|<\tau\right]>\alpha_T$ for some $\tau>0$ and some sequence $(x_n)$ in $X$ such that $x_n\to 0.$ In fact, if that were possible, we would take $\beta$ such that $\dlim_{n\to\infty}\mathbb{P}\left[\|T(x_n)\|<\tau\right]\ge\beta>\alpha_T,$ and the first part of the theorem would give us $\dlim_{x\to 0}\mathbb{P}\left[\|T(x)\|<\tau\right]\ge\beta>\alpha_T$ for all $\tau>0,$ a contradiction.
\end{proof}

\begin{corollary}\label{CorExi}
	Let $T:X\to \mathcal{R}_{\Omega}(Y)$ be a linear random operator. If there is a sequence $(x_n)$ in $X$ with $x_n\to 0$ such that 
    $$\dlim_{n\to\infty}\mathbb{P}\left[\|T(x_n)\|<\tau\right]\ge\alpha$$
	for all $\tau>0,$ then 
	$$\dlim_{x\to 0}\mathbb{P}\left[\|T(x)\|<\tau\right]\ge\alpha$$
	for all $\tau>0.$ Furthermore, if $\alpha=\alpha_T,$ holds equality on both limits.
\end{corollary}
\begin{proof}
	Analogously to the proof of the reciprocal of the Corollary \ref{eqseq1}, suppose that $T$ is probably not continuous with probability $\ge\alpha$, that is, that there are $\tau>0$ and $0<\varepsilon<\alpha$ such that, for every $\delta>0,$ there is $x\in X$ such that $\|x\|<\delta$ and $\mathbb{P}\left[\|T(x)\|<\tau\right]\le\varepsilon.$ As $x_n\to 0$, for each $\delta>0,$ we can choose $n_0\in \mathbb{N}$ such that $\|x_n\|<\delta$ for all $n>n_0$ and consequently  $\mathbb{P}\left[\|T(x_n)\|<\tau\right]\le\varepsilon,$ for all $n>n_0,$ so $T$ is probably not sequentially continuous, with probability $\ge\alpha$.
\end{proof}

The above corollary tells us that for a linear random operator to be probably continuous with probability $\ge\alpha$, it is sufficient to show that there is a sequence $(x_n)$ in $X, x_n\to 0,$ satisfying the limit $\dlim_{n\to\infty}\mathbb{P}\left[\|T(x_n)\|<\tau\right]\ge\alpha$
for all $\tau>0.$

\subsection{Random and Stochastic Closed Graph Theorem}

\begin{definition}
	Let $T:X\to \mathcal{R}_{\Omega}(Y)$ be a linear random operator and let $S(T)$ be the subspace separating $T,$ that is, 
	$$S(T)=\{y\in \mathcal{R}(Y): \exists x_n\to 0 \text{ with } T(x_n)\to y \text{ in probability} \}.$$
	We will say that $T$ has a probably closed graph, with probability $\ge\alpha,$ if $\mathbb{P}\left[y=0\right]\ge\alpha$ for all $y\in S(T).$ When $\alpha=1,$ we say that $T$ has a stochastically closed graph.
\end{definition}

The classical closed graph theorem (see, for example, \cite[Theorem 5.3.1]{Wilansky}) has a stochastic version that was presented without proof in \cite{TGF}. Next, we reformulate the statement presented in \cite{TGF} in order to obtain a demonstration of its veracity.

\begin{theorem}[Random/stochastic version of the Closed Graph Theorem]\label{vetgf}
	Let $T:X\to \mathcal{R}_{\Omega}(Y)$ be a linear random operator. Then $T$ is probably continuous with probability $\ge\alpha$ if and only if $T$ has probably closed graph, with probability $\ge\alpha$. 
\end{theorem}

\begin{proof}
	Suppose $T$ is probably continuous with probability $\ge\alpha$. If $y\in S(T),$ then there is a sequence $(x_n)$ in $X$ zuch that $x_n\to 0$ and $T(x_n)$ converges to $y$ in probability, that is,
	$$\dlim_{n\to \infty}\mathbb{P}\left[\|T(x_n)-y\|<\frac{\tau}{2}\right]=1 \text{ for all } \tau>0.$$
	On the other hand, since $T$ is probably continuous with probability $\ge\alpha$, the Corollary \ref{eqseq1} gives us 
	$$\dlim_{n\to \infty}\mathbb{P}\left[\|T(x_n)\|<\frac{\tau}{2}\right]\ge\alpha \text{ for all } \tau>0.$$ 
	As $\|y\|\le\|T(x_n)-y\|+\|T(x_n)\|$ and
	$$\mathbb{P}\left[\|T(x_n)-y\|<\frac{\tau}{2}, \|T(x_n)\|<\frac{\tau}{2}\right]\ge\mathbb{P}\left[\|T(x_n)-y\|<\frac{\tau}{2}\right]+\mathbb{P}\left[ \|T(x_n)\|<\frac{\tau}{2}\right]-1,$$ we have
	\begin{eqnarray*}
		\mathbb{P}\left[\|y\|<\tau\right]&\ge& \mathbb{P}\left[\|T(x_n)-y\|+ \|T(x_n)\|<\tau\right]
		\\ &\ge& \mathbb{P}\left[\|T(x_n)-y\|<\frac{\tau}{2}, \|T(x_n)\|<\frac{\tau}{2}\right]\\ &\ge& \mathbb{P}\left[\|T(x_n)-y\|<\frac{\tau}{2}\right]+\mathbb{P}\left[ \|T(x_n)\|<\frac{\tau}{2}\right]-1.
	\end{eqnarray*}
	Doing $n\to\infty,$ we get $\mathbb{P}\left[\|y\|<\tau\right]\ge 1+\alpha-1=\alpha,$ for all $\tau>0,$ so we conclude that $\mathbb{P}\left[y=0\right]\ge\alpha.$
	
	Now consider a sequence $(x_n)$ in $X$ such that $x_n\to 0,$ and $y\in\mathcal{R}_{\Omega}(Y)$ such that $(T(x_n))$ converges to $y$ in probability, that is, such that 
	$$\dlim_{n\to \infty}\mathbb{P}\left[\|T(x_n)-y\|<\tau\right]=1$$ 
	for all $\tau>0.$ Thus, $y\in S(T)$ and, by hypothesis, $\mathbb{P}\left[y=0\right]\ge\alpha.$ Hence, in view from the inequality $\|T(x_n)\|\le\|T(x_n)-y\|+\|y\|,$ we can write
	\begin{eqnarray*}
		\mathbb{P}\left[\|T(x_n)\|<\tau\right]
		&\ge&\mathbb{P}\left[\|T(x_n)-y\|+\|y\|<\tau\right]\\
		&\ge& \mathbb{P}\left[\|T(x_n)-y\|<\frac{\tau}{2}, \|y\|\le\frac{\tau}{2}\right]\\
		&\ge& \mathbb{P}\left[\|T(x_n)-y\|<\frac{\tau}{2}\right]+\left[ \|y\|\le\frac{\tau}{2}\right]-1,
	\end{eqnarray*}
	soon 
	$$\dlim_{n\to \infty}\mathbb{P}\left[\|T(x_n)\|<\tau\right]\ge 1+\alpha-1=\alpha,$$ and the Corollary \ref{CorExi} guarantees that $T$ is probably continuous with probability $\ge\alpha$.
\end{proof}

Of course, in the Theorem above, by virtue of the Corollary \ref{CorollaryE}, we can exchange $``\ge"$ for $``="$ and $``\alpha"$ for $``\alpha_T":$

\begin{corollary}
	Let $T:X\to \mathcal{R}_{\Omega}(Y)$ be a linear random operator. Then $T$ is probably continuous with probability $\alpha_T$ if and only if $T$ has a probably closed graph, with probability  $\alpha_T$. 
\end{corollary}

\centering\Large\text{ \textbf{CONFLICT OF INTEREST DECLARATION}} 

The author, Ph.D Kleber Soares Câmara, declares for the purposes that are made right that this article is devoid of any conflict of interest.

\centering\Large\text{ \textbf{DATA AVAILABILITY DECLARATION}}

The author, Ph.D Kleber Soares Câmara, declares, for the proper purposes, that all data referring to this article are fully available for publications of any nature, without any bonus for the author.

\author{C\^amara., K. S.}\newline\indent
\address{Departamento de Matem\'{a}tica,\newline\indent
	Universidade Federal Rural do Semi-árido,\newline\indent
	59.625-900, Mossoró -- RN, Brazil.}\newline\indent
\email{kleber.soares@ufersa.edu.br}

\end{document}